\documentclass[10pt]{article}

\pdfoutput=1

\usepackage[latin1]{inputenc}
\usepackage{mystyle,amssymb,dsfont,amsmath,amsthm}
\usepackage{url}
\graphicspath{{./figures/}}

\title{Iterative Bregman Projections\\for Regularized Transportation Problems}

\author{Jean-David Benamou\footnote{INRIA, MOKAPLAN, \texttt{jean-david.benamou@inria.fr}, \texttt{luca.nenna@inria.fr}} 
		\quad
		Guillaume Carlier\footnote{Ceremade, Universit\'e Paris-Dauphine, \texttt{\{carlier,peyre\}@ceremade.dauphine.fr}}
		\quad
		Marco Cuturi\footnote{Kyoto University, \texttt{mcuturi@i.kyoto-u.ac.jp}} \\
		Luca Nenna$^*$
		\quad
		Gabriel Peyré$^\dagger$}

\date{\today}

\renewcommand{\pi}{\gamma}
\renewcommand{\ga}{\epsilon}

\begin{document}

\maketitle



\begin{abstract}
	This article details a general numerical framework to approximate solutions to linear programs related to optimal transport. 
	The general idea is to introduce an entropic regularization of the initial linear program. This regularized problem corresponds to a  Kullback-Leibler Bregman divergence projection of a vector (representing some initial joint distribution) on the polytope of constraints. We show that for many problems related to optimal transport, the set of linear constraints can be split in an intersection of a few simple constraints, for which the projections can be computed in closed form. This allows us to make use of iterative Bregman projections (when there are only equality constraints) or more generally Bregman-Dykstra iterations (when  inequality constraints are involved). We illustrate the usefulness of this approach to several variational problems related to optimal transport: barycenters for the optimal transport metric, tomographic reconstruction, multi-marginal optimal transport and in particular its application to Brenier's relaxed solutions of incompressible Euler equations, partial un-balanced optimal transport and optimal transport with capacity constraints. 
\end{abstract}


\section{Introduction}

\subsection{Previous Works}
\label{sec-pw}
The theory of Optimal Transport (OT)~\cite{Villani03} defines a natural and useful geometry to compare measures supported on metric probability spaces. Its modern formulation as a linear program is due to Kantorovich~\cite{Kantorovich42}. 
OT has recently found a flurry of applications in various fields such as computer vision~\cite{RubTomGui00}, economy~\cite{MR2564439}~\cite{carlierekelandmatching}, computer graphics~\cite{Bonneel-displacement}, image processing~\cite{2014-xia-siims}, astrophysics~\cite{FrischNaturee}. 

\paragraph{Computational optimal transport.}

A major bottleneck that prevents the widespread of OT and its various generalizations is the lack of fast (possibly approximate) algorithms. 
Discrete optimal transport (\emph{i.e.} computing transport between sums of Diracs) reduces to a finite dimensional linear program. When the mass of each Dirac is constant and the two measures have the same number $N$ of Dirac masses, this problem reduces to an optimal matching problem, for which dedicated discrete optimization methods exist~\cite{Burkard09}, that roughly have $O(N^3)$ complexity, which is still computationally too demanding for most applications. 
Another line of research, initiated by~\cite{Benamou2000} relies on dynamic formulations, which corresponds to computing the transport as a geodesic, and can be re-casted as a convex optimization problem. We refer to~\cite{FPapPeyOud13} for an overview of several proximal optimization methods to tackle this problem. This requires adding an extra dimension (time variable along the geodesic) and is thus also computationally expensive. 
Semi-discrete optimal transport, i.e. optimal transport from a density to a  weighted sum of dirac masses is a classical strategy for a generalised version of the Monge-Amp\`ere equation, see \cite{Merigot:2011js} for recent improvements of this approach. Finally and for the quadratic ground cost optimal transport, a direct Newton solver approach to the non-linear Monge Amp\`ere equation can be used to compute density to density optimal transport. This holds under some regularity assumption on the densities and domain and hence on the transport map itself, see  \cite{Loeper:2005fn} and \cite{Benamou:2014jw} for instance.

\paragraph{Entropic regularization.}

A different approach consists in computing  a regularized version of the OT problem.  An interesting choice consists in penalizing the entropy of the joint coupling. This idea can be traced  back to Schrodinger~\cite{Shrodinger31} and can be related to the the so-called iterative proportional fitting procedure (IPFP)~\cite{DemingStephanIPFP} which has found numerous applications in the probability and statistics literature. We refer to~\cite{RuschendorfThomsen,LeonardShrodinger} for modern perspectives on this problem. Such a regularization also appears in the economy literature, where OT theory can be useful to predict flows of commodities or people in a market. In that context, regularizing the OT problem can also ensure the smoothness of such flows~\cite{wilson1969use,erlander1990gravity} or facilitate inference in matching models~\cite{Galichon-Entropic}. Such a regularization was also recently introduced in~\cite{CuturiSinkhorn} where it is shown that, in addition to favorable computational properties (parallelization, quadratic complexity) detailed below, such a regularization also yields a distance between histograms that can perform better in classification tasks than the usual OT distances.

The underlying idea of an entropic regularization is that entropy forces the solution to have a spread support, thus deviating from the fact that optimal couplings are sparse (\emph{i.e.} supported on a graph of a transport plan solving Monge's problem). A first impact of this regularization is that this non-sparsity of the solution helps to stabilize the computation. This can be related to the fact that entropic penalization defines a strongly convex program (as opposed to the initial OT problem) with a unique solution. 

Another (even more important) advantage of this entropic regularized OT problem is that its solution is a diagonal scaling of $e^{-C}$, the element-wise exponential matrix of $-C$, where $C$ is the ground cost defining the transport (see Section~\ref{sec-entropic-regul} for more details). The solution to this diagonal scaling problem can be found efficiently through the IPFP iterative scheme~\cite{DemingStephanIPFP}. This algorithm was later studied in detail by Sinkhorn in~\cite{Sinkhorn64,SinkhornKnopp67,Sinkhorn67} and its convergence proof was extended to continuous measures in~\cite{Ruschendorf95}.  

Entropic regularization of linear programs also shares some connection with interior point methods. These approaches make use of a $\log$-barrier function, which should be self-concordant to ensure a polynomial complexity for a given accuracy~\cite{NesterovNemirovskyBook}. Such a property does not hold for the entropic barrier, so it is not a competitive approach when it comes to approximating solutions of the original linear program by lowering the amount of regularization. As we advocate in this present paper, entropic regularization has however several other computational advantages (in particular because of its close connection with Kullback-Leibler projections), which makes it attractive when a slight amount of smoothing is acceptable in the computed approximation. 

\paragraph{Optimization using the Kullback-Leibler Divergence.}

When considering optimization over the simplex or the cone of positive vectors, it makes sense to replace the usual Euclidean metric by a divergence that quantifies with more relevance the difference between two vectors. Of particular interest for our work is the Kullback-Leibler (KL) divergence, since it is intimately related to entropic regularization. The simplest algorithmic block that can exploit such a divergence is the iterative projection on affine subsets of such cones under the KL divergence, which was introduced by Bregman~\cite{bregman1967relaxation}. Computing the projection on the intersection of generic convex sets requires to replace iterative projections by more complicated algorithms, such as for instance Dykstra's method~\cite{Dykstra83}. This algorithm is extended to Bregman divergences (such as KL) in~\cite{CensorReich-Dykstra} and a proof of convergence is given in~\cite{bauschke-lewis}. 
For references in probability and statistics that also address the case of continuous distributions see~\cite{csis,dyk, Ruschendorf95,bhatt2006}. 
Note that several other proximal algorithms have been extended to this setting~\cite{BauschkeCombettes-Dykstra}.

\paragraph{OT Barycenter.}

The OT metric has been extended in many ways. A natural extension is to consider the barycenter between several distributions (the case of only 2 measures defining the usual transport). Such a barycenter is defined as the solution of a convex variational problem (a weighted sum of OT distances) over the space of measures, which is studied in details in~\cite{Carlier_wasserstein_barycenter}. OT barycenters find applications for instance in statistics to define a mean empirical estimator from a family of observed histograms~\cite{BigotBarycenter}, or in machine learning~\cite{CuturiBarycenter} to provide an extended definition of $k$-means clustering and compute average histograms-of-features under the OT metric.

Solving this variational problem is challenging. Two recent numerical works have addressed this problem:~\cite{CuturiBarycenter} where a gradient descent on an entropic smoothing of OT distances is used and~\cite{Carlier-NumericsBarycenters} which is based on a dual formulation and tools from non-smooth optimization and computational geometry.

This barycenter problem can be extended to more complicated variational problem, such as for instance the Wasserstein propagation~\cite{Solomon-ICML}. Note that similar entropic regularization technics can be applied as well to this problem.

\paragraph{Multi-marginal transport.}

The OT barycenter problem, as introduced in~\cite{Carlier_wasserstein_barycenter}, is essentially equivalent to a multi-marginal optimal transport with quadratic cost as studied in~\cite{gansw}. Multimarginal reformulations are not usually not tractable, since they involve an optimization problem whose size grows exponentially with respect to the number of marginals.  Fortunately, the special structure of the OT barycenter problem leads to linear reformulations that are linear in the number of marginals,  see~\cite{Carlier-NumericsBarycenters} and Section~\ref{sec-barycenters} below. There are however applications where the problem under study intrinsically has a multi-marginal structure, that cannot be factorized as barycenter computations. 

Multi-marginal Optimal Transport,~\cite{MR2875651,Pass2}, is a natural extension of Optimal Transport with many 
potential fields of applications~: Economics~\cite{MR2564439}~\cite{carlierekelandmatching}, Density Functionnal Theory in Quantum Chemistry ~\cite{CPA:CPA21437}. The first important instance of multi-marginal transport was probably Brenier's generalised solutions of the Euler equations for incompressible fluids~\cite{BrenierEulerAMS,BrenierEulerARMA, BrenierEulerCPAM} which are clearly described in his review paper~\cite{BrenierGeneralized}.
Note that entropic regularization of the multimarginal transport problem leads to a problem of multi-dimensional matrix scaling~\cite{FranklinLorentz,BapatMultidim,RaghavanMultidim}.


\subsection{Contributions}

In this paper, we present a unified framework to numerically solve entropic approximations of several generalized optimal transport problems. This framework corresponds to defining appropriate entropic penalizations of the initial linear programs. The key idea is then to interpret the corresponding problems as projections of some input Gibbs density on an intersection of convex sets according to the Kullback-Leibler divergence. This problem can then be solved efficiently using either Bregman iterative projection (for intersection of affine spaces) or a more general Dykstra-like algorithm---these being well-known first order non-smooth optimization schemes. 
We investigate in details the applications of these ideas to several generalized OT problems: barycenter (Section~\ref{sec-barycenters}), tomographic reconstruction (Section~\ref{sec-tomography}), multi-marginal transport (Section~\ref{sec-multi-marginal}), partial transport (Section~\ref{sec-partial-ot}) and capacity constrained transport (Section~\ref{sec-capacity-ot}).
The code implementing the methods presented in this article can be found online\footnote{\url{https://github.com/gpeyre/2014-SISC-BregmanOT}}.


\subsection{Computational Speed}

The goal of this paper is to present a new class of efficient methods to provide approximate solutions to linear program generalizing OT. It is however important to realize that these methods become numerically unstable when the regularization parameter (denoted $\ga$ in the following) is small for two reasons: \emph{(i)} since some of the quantities manipulated in the proposed algorithms have an order of $e^{-1/\ga}$ (in particular Gibbs distributions denoted $\xi$ in the following) they become smaller than machine precision whenever the regularization $\ga$ is small; \emph{(ii)} more importantly, even if the first issue is taken care of by carrying out computations in the log domain, the convergence speed of iterative projection methods degrades significantly as $\ga \rightarrow 0$. We observe therefore that these methods are competitive in a range where the regularization term $\ga$ cannot be too small, and for which computed solutions exhibit a small amount of smoothing, see for instance Figure~\ref{fig-ot-sinkhorn} for a visual illustration of this phenomenon. 
It is thus not the purpose of this article to compare these new methods with more traditional ones (such as interior points or simplex), because they do not target the same problem.
Let us however single out the work of~\cite{CuturiBarycenter}, that solves the regularized barycenter problem described in Section~\ref{sec-barycenters} using a gradient descent scheme. In all our numerical experiments, we found however that the iterative Bregman projection converge with substantially computational effort than this gradient descent and, because they rely on alternate projections, do not require adjusting gradient step-sizes and are thus easier to deploy.

\subsection{Notations}

We denote the simplex in $\RR^N$
\eq{
	\Si_{N} \eqdef \enscond{ p \in \RR_+^N }{ \sum_i p_i = 1 }.
}
The polytope of couplings between $(p,q) \in \Si_N^2$ is defined as
\eq{
	\Pi(p,q) \eqdef \enscond{\pi \in \RR_+^{N \times N}}{ \pi \ones = p, \transp{\pi} \ones = q },
}
where $\transp{\pi}$ is the transpose of $\pi$ and $\ones \eqdef \transp{(1,\ldots,1)} \in \RR^N$.

For a set $\Cc$, we denote $\iota_\Cc$ its indicator, that is
\eq{
	\foralls x, \quad \iota_{\Cc}(x) = \choice{
		0 \qifq x \in \Cc, \\
		+\infty \quad\text{otherwise.}
	}
}

For $\pi \in \RR^{N \times N}$ for some $N > 0$, we define its entropy as
\eq{
	E(\pi) \eqdef -\sum_{i,j=1}^N \pi_{i,j} ( \log(\pi_{i,j}) - 1) + \iota_{\RR^+}(\pi_{i,j}), 
}
which is a concave function, where we used the convention $0\log(0)=0$.

The Kullback-Leibler divergence between $\pi \in \RR_+^{N \times N}$ and $\xi \in \RR_{++}^{N \times N}$ (i.e. $\xi_{i,j}>0$ for all $(i,j)$) is 
\eq{
	\KLdiv{\pi}{\xi} \eqdef \sum_{i,j=1}^N \pi_{i,j} \pa{ \log\pa{ \frac{\pi_{i,j}}{\xi_{i,j}} } - 1}.
}
With a slight abuse of notation, we extend these definitions for higher $d$-dimensional tensor arrays by replacing the sum over indices $(i,j)$ by sums over higher dimension indices. 

Given a convex set $\Cc \subset \RR^{N\times N}$, the projection according to the Kullback-Leibler divergence is defined as
\eq{
	\KLproj_\Cc(\xi) \eqdef \uargmin{ \pi \in \Cc } \KLdiv{\pi}{\xi}.
}

For vectors $(a,b) \in \RR^N \times \RR^N$, we denote entry-wise multiplication and division
\eql{\label{eq-entrywise}
	a \odot b \eqdef (a_i b_i)_i \in \RR^N
	\qandq
	\frac{a}{b} \eqdef (a_i/b_i)_i \in \RR^N.
}


\section{Iterative Bregman Projections and Dykstra Algorithm}

In this paper, we focus on regularized generalized OT problems that can be re-cast in the form
\eql{\label{proj-inter}
	\min_{\pi \in \Cc}  \KLdiv{\pi}{\xi}
}
where $\xi$ is a given point in $\RR_+^{N\times N}$, and $\Cc$ is an intersection of closed convex sets
\eq{
	\Cc = \bigcap_{\ell=1}^L \Cc_\ell
}
such that $\Cc$ has nonempty intersection with $\RR_+^{N\times N}$. 

In the following, we extend the indexing of the sets by $L$-periodicity, so that they satisfy
\eq{
	\foralls n \in \NN, \quad \Cc_{n+L} = \Cc_{n}.
}

\subsection{Iterative Bregman Projections}
\label{sec-iterative-bregman}

In the special case where the convex sets $\Cc_\ell$ are  affine subspaces (note that nonnegativity constraints are already in the definition of the entropy), it is possible to solve~\eqref{proj-inter} by simply using iterative KL projections. Starting from $\pi^{(0)} = \xi$, one computes
\eql{\label{eq-iter-bregmanproj}
	\foralls n>0, \quad
	\iter{\pi} \eqdef \KLproj_{\Cc_n}( \pi^{(n-1)} ). 
}
One can then show that $\iter{\pi}$ converges towards the unique solution of~\eqref{proj-inter},
\eq{  
	\iter{\pi} \to \KLproj_\Cc(\xi) \quad \mbox{as $n \to \infty$}.
}
see~\cite{bregman1967relaxation}.

\subsection{Dykstra's Algorithm}

When the convex sets $\Cc_\ell$ are not affine subspaces, iterative Bregman projections do not converge in general to the KL projection on the intersection. In contrast, Dykstra's algorithm~\cite{Dykstra83}, extended to the KL setting, does converge to the projection, see~\cite{bauschke-lewis}.

Dykstra's algorithm starts by initializing  
\eq{
	\pi^{(0)} \eqdef \xi
	\qandq
	q^{(0)}=q^{(-1)}= \dots =q^{(-L+1)} \eqdef \ones.
}
One then iteratively defines
\eql{\label{eq-iter-dystra}
	\iter{\pi} \eqdef  \KLproj_{\Cc_n}(\pi_{n-1} \odot q_{n-L}),
	\qandq
	\iter{q} \eqdef q^{(n-L)} \odot \frac{ \pi^{(n-1)} }{ \iter{\pi} }.
}
Recall here that $\odot$ and $\frac{\cdot}{\cdot}$ denotes entry-wise operations, see~\eqref{eq-entrywise}.

Dykstra algorithm converges to the solution of~\eqref{proj-inter}
\eq{  
	\iter{\pi} \to \KLproj_\Cc(\xi) \quad \mbox{as $n\to \infty$},
}
see~\cite{bauschke-lewis}.



\section{Entropic Regularization of Transport-like Problems}

\subsection{Entropic Optimal Transport Regularization}
\label{sec-entropic-regul}

To illustrate the class of methods developed in this paper, we first review a classical approach to optimal transport approximation, that we recast in the language of Kullback-Leibler projections. This allows us to recover well known results, but in a framework that is easily generalizable.

Following many previous works (see Section~\ref{sec-pw} for details) we consider the following discrete regularized transport
\eql{\label{eq-regul-transport}
	W_\ga(p,q) \eqdef \umin{\pi \in \Pi(p,q)} \dotp{C}{\pi} - \ga E(\pi).
}
The intuition underlying this regularization is that it enforces the optimal coupling $\pi_\ga^\star$ solution of~\eqref{eq-regul-transport} to be smoother as $\ga$ increases. This regularization also yields favorable computational properties since problem~\eqref{eq-regul-transport} is $\ga$-strongly convex. Its unique solution $\pi_\ga^\star$ can be obtained through elementary operations (matrix products, elementwise operations on matrices and vectors) as detailed below. If the optimal solution $\pi^\star$ of the (original, non-regularized, i.e. $\ga=0$) optimal transport problem is unique, then the optimal solution $\pi_\ga^\star$ of~\eqref{eq-regul-transport} converges to $\pi^\star$ as $\ga \rightarrow 0$. When other optimal solutions exist, $\pi_\ga^\star$ converges as $\ga \rightarrow 0$ to that with the largest entropy among those, again denoted $\pi^\star$. The convergence of minimizers of the regularized problem as $\ga\to 0$ is actually exponential
\eq{
	\norm{\pi_\ga^\star-\pi^\star}_{\RR^{N\times N}} \le M e^{-\la/\ga}
} 
where $\la$ and $M$ depend on $C$, $p$, $q$ and $N$   as  shown by Cominetti and San~Martin~\cite{CominettiAsympt}. 

Problem~\eqref{eq-regul-transport} can be re-written as a projection 
\eql{\label{eq-regul-ot-kl}
	W_\ga(p,q) = \ga \umin{\pi \in \Pi(p,q)} \KLdiv{\pi}{\xi}
	\qwhereq
	\xi = e^{ -\frac{C}{\ga} }
}
of $\xi$ according to the Kullback-Leibler divergence (here the exponential is computed component-wise). 

Problem~\eqref{eq-regul-ot-kl} can in turn be formulated as~\eqref{proj-inter} with $L=2$ affine subsets of $\RR_+^{N \times N}$
\eq{
	\Cc_1 \eqdef \enscond{\pi \in \RR_+^{N \times N}}{ \pi \ones = p }
	\qandq
	\Cc_2 \eqdef \enscond{\pi \in \RR_+^{N \times N}}{ \transp{\pi} \ones = q }.
}
The application of Bregman iterative projection (detailed in Section~\ref{sec-iterative-bregman}) to this splitting corresponds to the so-called IPFP/Sinkhorn algorithm (see Section~\ref{sec-pw} for bibliographical details).

The following well-known proposition details how to compute the relevant projections.

\begin{prop}\label{prop-projkl-row-cols}
	One has, for $\bar\pi \in (\RR_+)^{N \times N}$, 
	\eql{\label{eq-proj-row-cols}
		\KLproj_{\Cc_1}(\bar\pi) = \diag\pa{ \frac{p}{ \bar\pi\ones } } \bar\pi
		\qandq
		\KLproj_{\Cc_2}(\bar\pi) =  \bar\pi \diag\pa{ \frac{q}{ \transp{\bar\pi} \ones } }
	}
\end{prop}

In plain words, the two projections in equation~\eqref{eq-proj-row-cols} normalize (with a multiplicative update) either the rows or columns of $\bar\pi$ so that they have the desired row-marginal $p$ or column-marginal $q$.

\begin{rem}[Fast implementation]\label{rem-fast-sink}
An important feature of iterations~\eqref{eq-iter-bregmanproj}, when combined with projections~\eqref{eq-proj-row-cols} is that the iterates $\iter{\pi}$ satisfy
\eq{
	\iter{\pi} = \diag(\iter{u})\xi\diag(\iter{v})
} 
where the vectors $(\iter{u},\iter{v}) \in \RR^N \times \RR^N$ satisfy $v^{(0)}=\ones$ and obey the recursion formula
\eq{
	\iter{u} = \frac{p}{\xi \iter{v}}
	\qandq
	v^{(n+1)} = \frac{q}{ \transp{\xi} \iter{u}}.
}
This allows to implement this algorithm by only performing matrix-vector multiplications using a fixed matrix $\xi$, possibly in parallel if several OT are to be computed for several marginals sharing the same ground cost $C$ as shown in~\cite{CuturiSinkhorn}.
\end{rem}

\newcommand{\myfig}[2]{\imgbox{\includegraphics[width=.155\linewidth]{transport/#1/#1-#2}}}

\begin{figure}[h!]
	\centering
	\begin{tabular}{c}
	\includegraphics[width=.5\linewidth]{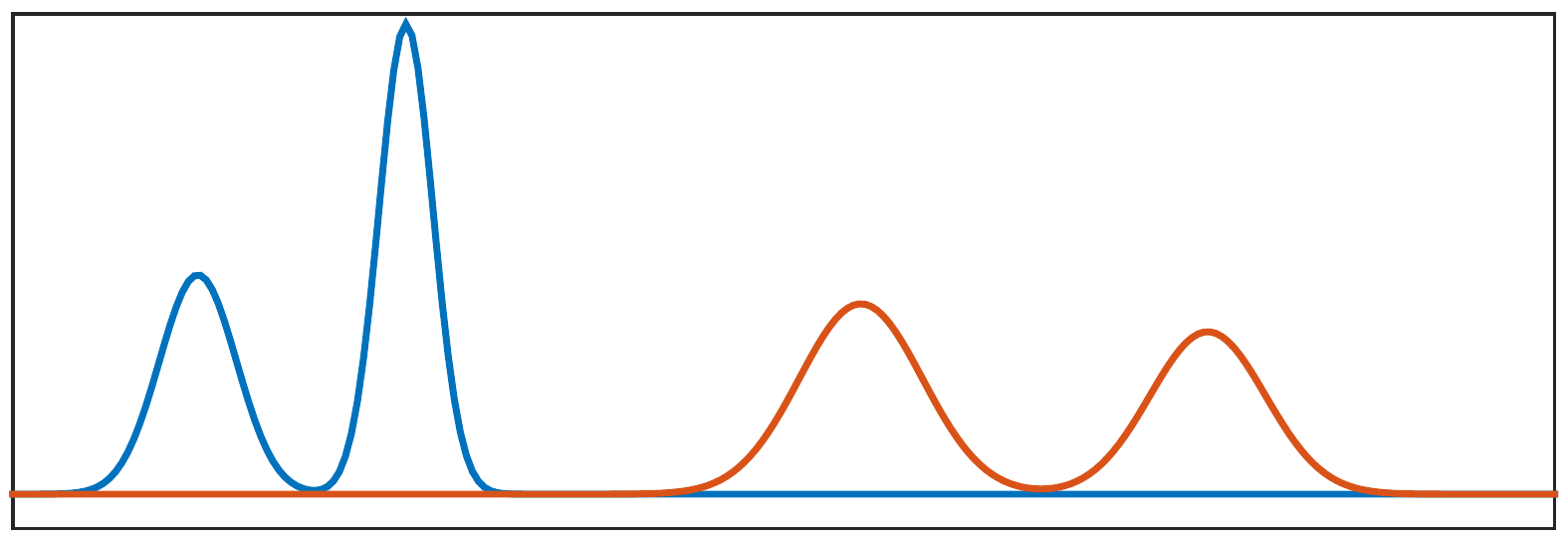}\\
	Marginals ${\color{blue} p}$ and ${\color{red} q}$\\[2mm]
	\end{tabular}
	\begin{tabular}{@{}c@{\hspace{1mm}}c@{\hspace{1mm}}c@{\hspace{1mm}}c@{\hspace{1mm}}c@{\hspace{1mm}}c@{}}
		\myfig{iteration}{1}&
		\myfig{iteration}{4}&
		\myfig{iteration}{10}&
		\myfig{iteration}{40}&
		\myfig{iteration}{100}&
		\myfig{iteration}{1000} \\
		$\ell=1$ & $\ell=4$ & $\ell=10$ & $\ell=40$ & $\ell=100$ & $\ell=1000$ \\[2mm]
		\myfig{regularization}{4}&
		\myfig{regularization}{6}&
		\myfig{regularization}{10}&
		\myfig{regularization}{20}&
		\myfig{regularization}{40}&
		\myfig{regularization}{60} \\
		$\ga=3/N$ & $\ga=6/N$ & $\ga=10/N$ & $\ga=20/N$ & $\ga=40/N$ & $\ga=60/N$
	\end{tabular}
	\caption{%
		\textit{Top:} the input densities $p$ (blue curve) and $q$ (red curve).
		\textit{Center:} evolution of the couplings $\pi^{(\ell)}$ at iteration $\ell$ of the Sinkhorn algorithm. 
		\textit{Bottom:} solution $\pi=\pi_\ga^\star$ of~\eqref{eq-regul-ot-kl} for several values of $\ga$.
	}
   \label{fig-ot-sinkhorn}
\end{figure}

Figure~\ref{fig-ot-sinkhorn} displays examples of transport maps $\pi=\pi_\ga^\star$ solving~\eqref{eq-regul-ot-kl}, for two 1-D marginals $(p,q) \in \RR^{N} \times \RR^N$ discretizing continuous densities on a uniform grid $(x_i)_{i=1}^N$ of $[0,1]$, and with a ground cost $C_{i,j} = \norm{x_i-x_j}^2$. 
The computation is performed with $N = 256$.
This figure shows how $\pi_\ga^\star$ converges towards a solution of the original un-regularized transport as $\ga \rightarrow 0$. It also shows how the iterates of the algorithm $\iter{\pi}$ progressively shift mass away from the diagonal during the iterations.


\subsection{Optimal Transport Barycenters}
\label{sec-barycenters}

We are given a set $(\km{p})_{k=1}^K$ of input marginals $\km{p} \in \Si_N$, and we wish to compute a weighted barycenter according to the Wasserstein metric. This problem finds many applications, as highlighted in Section~\ref{sec-pw}.

Following~\cite{Carlier_wasserstein_barycenter}, the general idea is to define the barycenter as a solution of a variational problem mimicking the definition of barycenters in Euclidean spaces. Given a set of normalized weights $\la \in \Si_K$, 
we consider the problem
\eql{\label{eq-barycenter-regul}
	\min_{p \in \Si_N} \enscond{ \sum_{k=1}^K \la_k W_\ga( \km{p},p ) }{ p \in \Si_N }
}
which as in \cite{Carlier-NumericsBarycenters}  is re-written as 
\eq{
	\foralls k = 1,\ldots,K, \quad p = \pi_{k} \ones
}
where the set of optimal couplings $\bpi = (\pi_{k})_{k=1}^K \in (\RR_+^{N \times N})^K$ solves
\eql{\label{eq-bary-variational-kl}
	\min \enscond{
		\KLdivL{\bpi}{\bxi} \eqdef \sum_{k=1}^K \la_k \KLdiv{ \km{\pi} }{\km{\xi}}
	}{
		\bpi \in \Cc_1 \cap \Cc_2	
	}
}
\eq{
	\qwhereq
	\forall k, \quad \km{\xi} \eqdef \xi \eqdef e^{ -\frac{C}{\ga} }
}	
and  the constraint sets are defined by
\begin{align*}
	\Cc_1 &\eqdef \enscond{ \bpi = (\km{\pi})_k \in (\Si_N)^K }{ \foralls k, \: \pi_{k}^T \ones = \km{p} } \\
	\qandq 
	\Cc_2 &\eqdef \enscond{ \bpi = (\km{\pi})_k \in (\Si_N)^K }{ \exists p \in \RR^N, \foralls k, \: \pi_k \ones = p }
\end{align*}

It is easy to check that the Bregman iterative projection scheme can be applied to this setting by simply replacing $\KL$ by $\KL_\la$.

The $\KL_\la$ projection on $\Cc_1$ is computed as detailed in Proposition~\ref{prop-projkl-row-cols}, since it is equal to the $\KL$ projection of each $\km{\xi}=\xi$ on a constraint of fixed marginal $\km{p}$. 
The $\KL_\la$ projection on $\Cc_2$ is computed as detailed in the following proposition.

\begin{prop}\label{prop-klproj-bary}
	For $\bar\bpi \eqdef (\km{\bar\pi})_k \in (\RR_+^{N \times N})^K$, 
	the projection $\bpi \eqdef (\km{\pi})_{k=1}^K = \KLprojL_{\Cc_2}(\bar\bpi)$ satisfies
	\eql{\label{eq-projc2-bary}
		\foralls k, \quad
		\km{\pi} = \diag\pa{ \frac{p}{ \km{\bar\pi} \ones } } \km{\bar\pi}
		\qwhereq
		p \eqdef \prod_{r=1}^K ( \bar\pi_r \ones )^{\la_r}
	} 
	where $\prod$ and $(\cdot)^{\la_r}$ should be understood as entry-wise operators.
\end{prop}

\begin{proof}
	Introducing the variable $p$ such that for all $k$, $\pi_k \ones = p$, the first order conditions of the projection 
	$\KLprojL_{\Cc_2}(\bar\bpi)$ states the existence of Lagrange multipliers $(u_k)_k$ such that 
	\eq{
		\foralls k, \quad \la_k \log\pa{ \frac{\km{\pi}}{\km{\bar\pi}} } + u_k \transp{\ones} = 0 
		\qandq
		\sum_r u_r = 0. 
	}
	Denoting $a_k = e^{-u_k}$, one has $\prod_k a_k=\ones$ and $\km{\pi}=\diag(a_k^{1/\la_k}) \km{\bar\pi}$.
	Condition $\pi_k \ones = p$ thus implies that 
	\eq{
		a_k = \pa{ \frac{p}{\km{\bar\pi} \ones} }^{\la_k}, 
	}
	and condition $\prod_k a_k=\ones$ gives the desired value~\eqref{eq-projc2-bary} for $p$.	
\end{proof}

\begin{rem}[Special case]
	Note that when $K=2$, $(\la_1,\la_2) = (0,1)$, one retrieves exactly the IPFP/Sinkhorn algorithm to solve the entropic OT, as detailed in Section~\ref{sec-entropic-regul}. Our novel scheme to compute barycenters should thus be understood as the natural generalization of this IPFP algorithm to barycenters.  
\end{rem}

\begin{rem}[Memory efficient and parallel implementation]
	Similarly as for Remark~\ref{rem-fast-sink}, one verifies that iterations~\eqref{eq-iter-bregmanproj} in the special case of problem~\eqref{eq-bary-variational-kl} leads to iterates $\iter{\bpi} = ( \pi_k^{(n)} )_k$ which satisfy, for each $k$
	\eq{
		\pi_k^{(n)} = \diag(u_k^{(n)}) \xi \diag(v_k^{(n)})
	}
	for two vectors $(u_k^{(n)},v_k^{(n)}) \in \RR^{N} \times \RR^N$ initialized as $v_k^{(0)} = \ones$ for all $k$, 
	and computed with the iterations
	\eq{
		u_k^{(n)} = \frac{p^{(n)}}{\xi v_k^{(n)}}
		\qandq
		v_k^{(n+1)} = \frac{\km{p}}{ \transp{\xi} u_k^{(n)}}
	}
	where $p^{(n)}$ is the current estimate of the barycenter, computed as
	\eq{
		p^{(n)} = \prod_{k=1}^N \pa{ u_k^{(n)} \odot (\xi v_k^{(n)}) }^{\la_k}.
	}
	A nice feature of these iterations is that they can be computed in parallel for all $k$ using multiplications between the matrix $\xi$ and matrices storing $(u_k^{(n)})_k$ and $(v_k^{(n)})_k$ as columns. 
\end{rem}

Figure~\ref{fig-barycenter-entropic} shows an example of barycenters computation for $K=3$.	The three vertices of the triangle show the input densities $(p_1,p_2,p_3)$ which are uniform on binary shapes (diamond, annulus and square). 
The other points in the triangle display the results for the following values of $\la$
\eq{
	\begin{tabular}{c}
		(0,0,1) \\
		(1, 0, 3)/4 \quad (0, 1, 3)/4 \\
		(1,0,1)/2 \quad (1,1,2)/4 \quad (0,1,1)/2 \\ 
		(3,0,1)/4 \quad (2,1,1)/4 \quad (1,2,1)/4 \quad (0,3,1)/4 \\
		(1,0,0)\quad (3,1,0)/4 \quad (1,1,0)/2 \quad (1,3,0)/4 \quad (0,1,0)
	\end{tabular}
} 
The computation is performed on an uniform 2-D grid of $N = 256 \times 256$ points in $[0,1]^2$, and $\ga=2/N$.

\begin{figure}[h!]
	\centering
	\includegraphics[width=.7\linewidth]{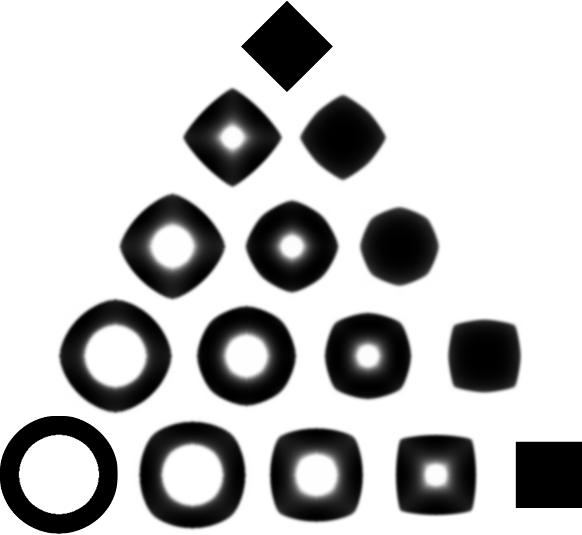}
	\caption{%
		Example of OT barycenters with entropic smoothing. 
	}
   \label{fig-barycenter-entropic}
\end{figure}


\subsection{Partial Radon Inversion with OT Fidelity}
\label{sec-tomography}

The partial Radon transform (i.e. the computation of integrals of the data along parallel rays in a small limited set of directions) is a mathematical model for several scanning medical acquisition devices. This is an ill-posed linear operator, and inverting it while preventing noise and artifacts to blowup is of utmost importance for the targeted imaging applications. It is out of the scope of this paper to review the overwhelming literature on the topic of Radon inversion, and we refer to the book~\cite{HermanTomography} for an overview of classical approaches, and~\cite{Klann11} and the references therein for examples of state-of-the art methods. 

The goal of this section is not to present a state of the art inversion scheme, but rather to show how the method recently introduced by~\cite{AbrahamRadon} can be solved using a simple iterative Bregman projection algorithm. We describe here the method in a fully discretized setting, where the Radon transform is implemented using a nearest neighbor interpolation. 

We consider a square discretization grid of $N = N_0 \times N_0$ pixels, indexed with 
\eq{
	s = (s_1,s_2) \in \Om_N \eqdef \om_{N_0} \times \om_{N_0}
	\qwhereq
	\om_{N_0} \eqdef \{1,\ldots,N_0\}^2.
}

Given an angle $\th$, we consider the following discrete lines in $\Om_N$, $\foralls (s_1,s_2) \in \Om_N$, 
\eq{
	\ell_{s_1,s_2}^\th \eqdef 
	\choice{
		(s_2, s_1 + (s_2-1) \tan(\theta)  \text{ mod } N_0), 
			\;\text{if}\; 
			\text{mod}(\th,\pi) \in [0,\pi/4] \cup [3\pi/4,\pi] \\
		(s_1 + (s_2-1) \tan(\theta)  \text{ mod } N_0, s_2), 
			\;\text{if}\; 
			\text{mod}(\th,\pi) \in [\pi/4,3\pi/4] \\	
	}
}
where the mod $N_0$ is a modulo $N_0$ that maps the indices in the admissible range $\Om_N$, hereby effectively implementing a convenient cyclic boundary condition. 

The discrete Radon integration $R_{\th}(f) \in \RR^{N_0}$ for such an angle $\th$ of an image $f \in \RR^N \sim \RR^{N_0 \times N_0}$ computes
\eq{
	\foralls s_1 \in \om_{N_0}, \quad
	R_\th(f)_{s_1} \eqdef \sum_{s_2=1}^{N_0} f_{\ell_{s_1,s_2}} \in \RR.
}
Its adjoint $R_{\th}^*$, the back-propagation operator along direction $\th$, reads, for $r \in \RR^{N_0}$, 
\eq{
	f = R_{\th}^*(r)  \qwhereq
	\foralls s \in \Om_N, \quad
	f_{\ell_{s_1,s_2}} = r_{s_1}.
}

We consider the linear inverse problem of reconstructing an approximation of an unknown image $f_0 \in \RR^N$ from (possibly noisy) partial Radon measurements 
\eql{\label{eq-radon-fwd}
	r = (\km{r})_{k=1}^K
	\qwhereq
	\km{r} = R_{\th_k}(f_0) + \km{w}, 
}
where $(\th_k)_{k=1}^K$ is a set of angles, and $\km{w}$ is the noise perturbing the observation. We denote $R(f) = (R_{\th_k}(f))_{k=1}^K$ so that $R : \RR^N \mapsto \RR^{N_0 \times K}$ is a linear operator, and $R^*$ is its adjoint.

The simplest way to perform a reconstructiong is to solve for the least squares  estimate 
\eql{\label{eq-l2-reconstr-radon}
	R^+(r) = R^* (RR^*)^{-1} r \eqdef \uargmin{f} \norm{f}
		\qstq \foralls k,  \quad R_{\th_k}(f) = \km{r}. 
}
This linear inverse does a poor job in the case of a small number $K$ of projections, since it exhibits reconstruction artifacts, as shown on Figure~\ref{fig-radon}.

In order to obtain a better reconstruction, and following the idea introduced in~\cite{AbrahamRadon}, we suppose that one has access to a template $g_0 \in \RR^N$ which is intended to be some approximation of $f_0$, possibly with some translation and small deformations. To leverage the strong robustness of the OT distance with respect to translation and small deformations, the reconstruction is obtained by using a sum of Wasserstein distances to the observation and to the template
\eql{\label{eq-radon-recons-initi}
	\umin{f \in \Si_N} \la_1 W_{2,\ga}( f,g_0 ) + \la_2\sum_{k=1}^Q W_{1,\ga}( R_{\th_k}(f),\km{r} ),
}
where $0 < \la_1 \leq 1$ is a weight that accounts for the degree of confidence in the template $g_0$, and $\la_2=1-\la_1$.

Here, $W_{2,\ga}$ indicates the entropic Wasserstein distance~\eqref{eq-regul-ot-kl} on the 2-D grid $\Om_N$ where we defined the cost matrix $C=C^2$ as
\eq{
	\foralls (s, t) \in (\Om_N)^2, \quad
		C_{s,t}^2 \eqdef (s_1-t_1)^2 + (s_2-t_2)^2
}
and $W_{1,\ga}$ indicate the entropic Wasserstein distance~\eqref{eq-regul-ot-kl} on a 1-D periodic grid for the cost $C=C^1$
\eq{
	\foralls (i,j) \in (\om_{N_0})^2, \quad
		C_{i,j}^1 \eqdef \umin{k \in \ZZ} (i-j+k N_0)^2.
}

Similarly to the barycenter problem~\eqref{eq-bary-variational-kl}, we compute $f$ solving~\eqref{eq-radon-recons-initi} as $f = \pi \ones_N$ where $\pi$ solves 
\eql{\label{eq-radon-variational-kl}
	\min
	\enscond{
		\la_1 \KLdiv{\pi}{\xi} +  \la_2\sum_{k=1}^K \KLdiv{\km{\pi}}{\km{\xi}}
	}{
		\foralls k, \; (\pi,\km{\pi}) \in \Cc_k, \cap \tilde\Cc_k
	}
}
\eq{
	\qwhereq
	\xi = e^{ -\frac{ C^2 }{\ga} }
	\qandq
	\foralls k, \quad 
	\km{\xi} = e^{ -\frac{C^1}{\ga} }	
}
(here the $\exp$ should be understood component-wise) and where we introduced
\begin{align*}
	\foralls k, \; \Cc_k &= 
		\enscond{ (\pi,\km{\pi}) \in \Si_{N} \times \Si_{N_0} }{ 
			\transp{\pi} \ones_N = g_0 \qandq \pi_k^{*} \ones_{N_0} = \km{r} } \\
	\foralls k, \; \tilde\Cc_k &= 
		\enscond{ (\pi,\km{\pi}) \in \Si_{N} \times \Si_{N_0} }{
			R_{\th_k}(\pi \ones_{N}) = \km{\pi} \ones_{N_0}
	}.
\end{align*}
Problem~\eqref{eq-radon-variational-kl} thus corresponds to a (weighted) KL projection on the intersection of $2K$ constraints. Computing the projection on each $\Cc_k$ is achieved as detailed in Proposition~\ref{prop-projkl-row-cols}. The following proposition, which is a simple extension of Proposition~\ref{prop-klproj-bary}, shows how to project on $\tilde\Cc_k$.

\begin{prop}
	For any $k = 1,\ldots,K$, the projection 
	$\bpi = (\pi,\km{\pi}) = \KLprojL_{\tilde\Cc_k}(\bar\bpi)$ of $\bar\bpi = (\bar\pi,\km{\bar\pi})$
	for the KL metric 
	\eq{
		\KLdivL{\bpi}{\bar\bpi} \eqdef \la_1 \KLdiv{\pi}{\bar\pi} + \la_2 \KLdiv{\km{\pi}}{\km{\bar\pi}}
	}
	satisfies
	\eq{
		\pi = \diag\pa{ R_{\th_k}^* \pa{ \frac{\de_k}{\al_k} } }  \bar\pi
		\qandq
		\km{\pi} = \diag\pa{ \frac{\de_k}{\be_k} } \km{\bar\pi}
	}
	where we defined
	\eq{
		\km{\al} \eqdef R_{\th_k}(\bar\pi \ones_{N}), \quad
		\km{\be} \eqdef \km{\bar\pi} \ones_{N_0}, 
		\qandq
		\km{\de} = \al_k^{\la_1} \odot \be_k^{\la_2}
	}
	where the exponentiations are component-wise.
\end{prop}

Figure~\ref{fig-radon} shows an example of application of the method for an image $f_0$ which is discretized on a grid of $N=80 \times 80$ points, using $\ga=2/N$ and $K=12$ Radon directions. There is no additional noise in the measurements, i.e. $\km{w}=0$ in~\eqref{eq-radon-fwd}. The template $g_0$ is a binary disk. These results show how using a large $\la_1$ recovers a result that is close to $g_0$, while using a smaller $\la$ introduces the geometric features of $f_0$ but also reconstruction artifacts. Note that the linear reconstruction $R^+(r)$ contains reconstruction artifacts.

\newcommand{\myfigRad}[1]{\includegraphics[width=.24\linewidth]{radon/#1}}
\newcommand{\myfigRadMarg}[1]{\includegraphics[width=.48\linewidth]{radon/marginals/marginal-#1}}

\begin{figure}[h!]
	\centering
	\begin{tabular}{@{}c@{\hspace{1mm}}c@{\hspace{1mm}}c@{}}
		\myfigRad{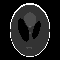} &
		\myfigRad{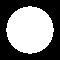} &
		\myfigRad{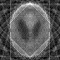} \\
		$f_0$ & $g_0$ & $R^+(r)$ \\[2mm]
	\end{tabular}
	\begin{tabular}{@{}c@{\hspace{1mm}}c@{\hspace{1mm}}c@{\hspace{1mm}}c@{}}
		\myfigRad{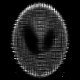} &
		\myfigRad{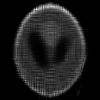} &
		\myfigRad{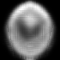} &
		\myfigRad{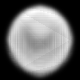} \\
		$\la_1=0.1$ & $\la_1=0.5$ & $\la_1=0.9$ & $\la_1=0.99$
	\end{tabular}
	\begin{tabular}{@{}c@{\hspace{1mm}}c@{}}
		\myfigRadMarg{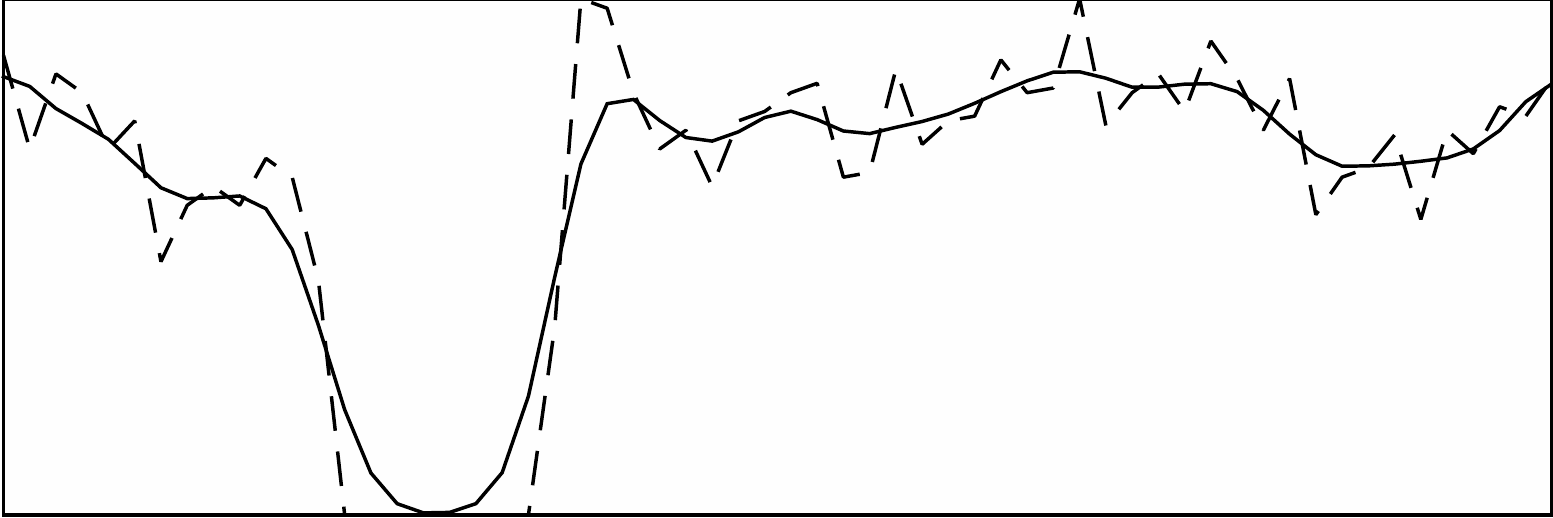} &
		\myfigRadMarg{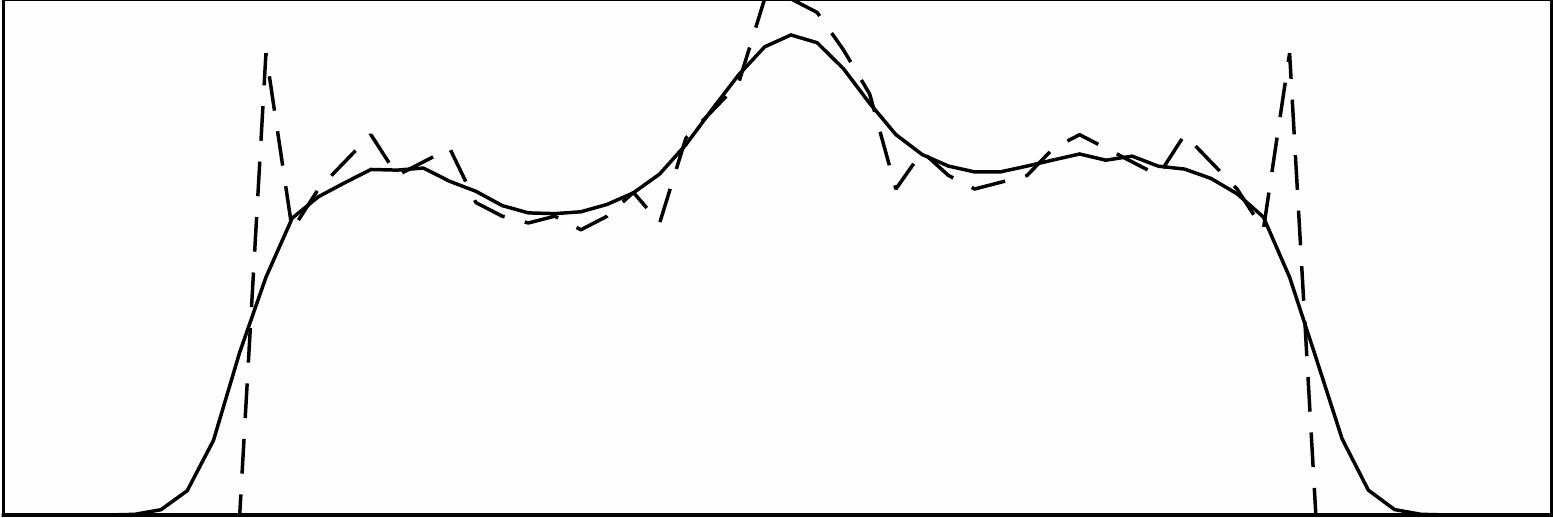} \\
		$k=1$ & $k=3$ \\
		\myfigRadMarg{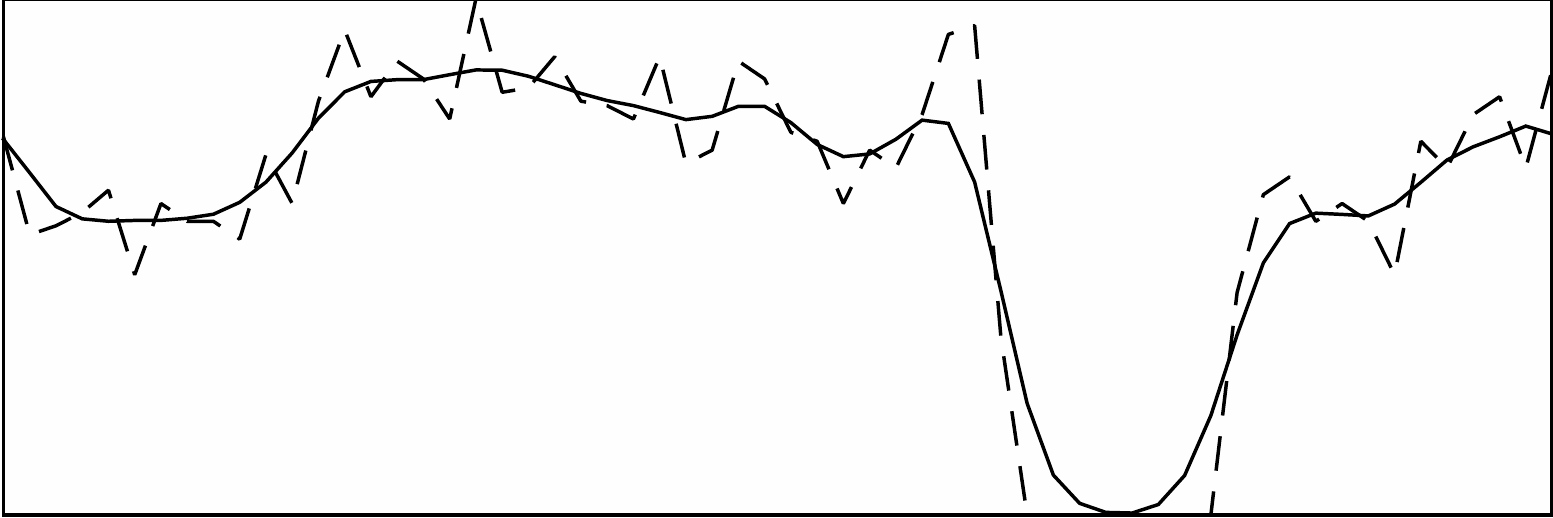} &
		\myfigRadMarg{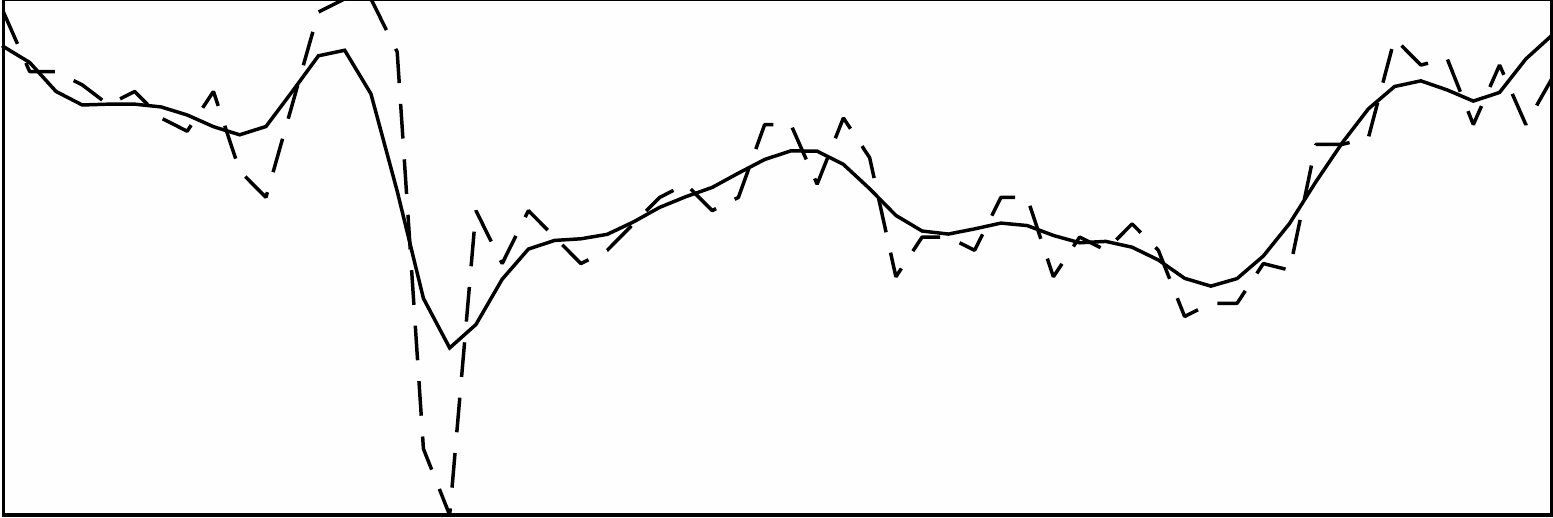} \\
		$k=5$ & $k=7$ 
	\end{tabular}	
	\caption{%
		Example of reconstructions from partial Radon measurements. 
		The two bottom rows show the input transforms $r_k=R_{\th_k}(f_0)$ (dashed curves)
		and the recovered Radon transforms $R_{\th_k}(f_0)$ where $f$ solved~\ref{eq-radon-recons-initi} (plain curves),
		for a few values of $k$. 
	}
   \label{fig-radon}
\end{figure}


\section{Multi-marginal  Optimal Transport}
\label{sec-multi-marginal}

Multi-marginal optimal transport is a natural extension of optimal transport with many 
potential fields of applications, see Section~\ref{sec-pw}.

\subsection{Multi-marginal Regularization}
\label{subsec-multi-marginal-regularization} 
As in the barycenter problem, we are given $K$ marginals $(\km{p})_{k=1}^K$. 
We denote $\pi \in \RR_+^{N^K}$ a $K$-dimensional array, indexed as $\pi_j$ for $j=(j_1,\ldots,j_K) \in \{1\ldots,N\}^K$. 
The push-forward $S_k(\pi) \in \RR^N$ of such a $\pi$ along dimension $k$ is computed as
\eq{
	\foralls i \in \{1,\ldots,N\}, \quad
	S_k(\pi)_i \eqdef \sum_{j_1,j_2,\ldots,j_{k-1},j_{k+1}, ..,j_K} \pi_{j_1,j_2,\ldots,j_{k-1},i,j_{k+1}, ..,j_K}.
}
The set of couplings between the marginals is
\eq{
	\Pi(p_1,p_2,\ldots,\km{p}) \eqdef 
	\enscond{ \pi \in \RR_+^{N^K} 
	}{ 
		\forall k, \quad S_k(\pi) = \km{p}
	}.
}

Given a cost matrix $C \in \RR_+^{N^K}$, the regularized OT problem~\eqref{eq-regul-transport} is generalized to this multi-marginal setting as
\eql{\label{eq-multimarg-entropic}
	\umin{\pi \in \Pi(p,q)} \dotp{C}{\pi} - \ga E(\pi).
}

Similarly as~\eqref{eq-regul-ot-kl}, this problem can be re-cast as a KL projection  
\eql{\label{eq-multimarginal-proj}
	\min_\pi \enscond{
		\KLdiv{\pi}{\xi} 
	}{
		\pi \in \Cc_1 \cap \Cc_2 \cap \ldots \cap \Cc_K
	}
	\qwhereq
	\xi \eqdef e^{-\frac{C}{\ga}}
}
where the exponentiation is exponent-wise, and where
\eq{
	\Cc_k \eqdef \enscond{ \pi \in \RR_+^{N^K} }{ S_k(\pi)=\km{p} }.
}

The Bregman projection on each of  the convex $\Cc_k$ are again given by a simple normalisation as detailed in the following proposition.

\begin{prop}
	For any $k$, denoting $\pi = \KLproj_{\Cc_k}(\bar{\pi})$, one has 
	\eq{
		\foralls j=(j_1,\ldots,j_K), \quad		
		\pi_{j} = \frac{(\km{p})_{j_k}}{ S_k(\bar\pi)_{j_k} } \bar\pi_{j} 
	}
\end{prop}

It is thus possible to use the Bregman iterative projection detailed in Section~\ref{sec-iterative-bregman} to compute the projection~\eqref{eq-multimarginal-proj}. 

We now detail in the two following sections two typical cases of application of multi-marginal OT: grid-free barycenter computation (Section~\ref{sec-multimarg-bary}) and resolution of generalized Euler flow (Section~\ref{sec-gen-euler}).

\subsection{Multi-marginal Barycenters}
\label{sec-multimarg-bary}

As shown in~\cite{Carlier_wasserstein_barycenter}, the computation of barycenters of measures (thus the continuous analogous of~\eqref{sec-barycenters}) can be computed by solving a multi-marginal transport problem. 



Let us suppose that the input measures $(\km{\mu})_{k=1}^K$ defined on $\RR^d$ are of the form $\km{\mu} = \sum_{i=1}^N p_{k,i} \de_{x_i}$, where $\km{p}=(p_{k,i})_{i=1}^N \in \Si_N$,  where $\{x_i\}_i \subset \RR^d$ and $\de_{x}$ is the Dirac measure at location $x \in \RR^d$. It is shown in~\cite{Carlier_wasserstein_barycenter} that the Wasserstein barycenter of the $\km{\mu}$ with weights $(\la_k)_k \in \Si_K$ for the quadratic Euclidean distance ground cost is 
\eql{\label{eq-continuous-barycenter}
	\mu_\la \eqdef \sum_{j=(j_1,j_2,\ldots,j_K)} 
		 \pi_{j} \de_{ A_{j}(x)}
}
where $A_{j}(x)\eqdef \sum_k \la_k x_{j_k}$ is the Euclidean barycenter and $\pi \in \RR_+^{N^K}$ is an optimal multi-marginal coupling that solves~\eqref{eq-multimarg-entropic} for the following cost 
\eq{
	C_{j}  = \sum_{1 \leq k \leq K}\dfrac{\lambda_{k}}{2} \norm{ x_{j_k}-A_j(x) }^2.  
}

An important point to note is that the measure barycenter~\eqref{eq-continuous-barycenter} is in general composed of more than $N$ Diracs, and that these Diracs are not constrained to be on the discretization grid $(x_i)_i$. In particular, the obtained result is different from the one obtained by solving~\eqref{eq-barycenter-regul}, which computes a barycenter that lies on the same grid as the input measures. In some sense, formulation~\eqref{eq-continuous-barycenter} is able to compute the ``true'' barycenter of measures, whereas~\eqref{eq-barycenter-regul} computes an approximation on a fixed grid, but the price to pay is the resolution of a high-dimensional multi-marginal program.  

Figure~\ref{fig:barm1} shows an histogram depiction of the measure $\mu_\la$ defined in~\eqref{eq-continuous-barycenter}, for the iso-barycenter (i.e. $\la_k=1/K$ for all $k$). It is computed by first solving~\eqref{eq-multimarg-entropic} with the same three marginals used in Figure~\ref{fig-barycenter-entropic}. The histogram $p \in \Si_N$ computed on a grid of $N=60 \times 60$ points. Each $p_i$ is the total mass of $\mu_\la$ in the discretization square $S_i$ of size $1/\sqrt{N} \times 1/\sqrt{N}$, i.e. $p_i = \mu_\la(S_i)$. 


\begin{figure}[htdp]
	\centering
	\imgbox{\includegraphics[width=.45\textwidth]{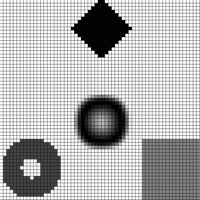}}
	\caption{Barycenter computed by solving the multi-marginal problem with three marginals (annulus, diamond and square) discretized on an uniform 2-D grid of $N=60 \times 60$ points in $[0,1]^{2}$ and $\ga=0.005$. See the main text body for details about how the display of the barycenter measure is performed. }
  	\label{fig:barm1}
\end{figure}

\subsection{Generalized Euler Flows}
\label{sec-gen-euler}

Brenier proposed in a series of papers~\cite{BrenierEulerAMS,BrenierEulerARMA,BrenierEulerCPAM} a relaxation of the Euler equation of incompressible fluids with constrained initial and  final data. These data are conveniently expressed as a volume preserving map $\Xi$ of the domain. This relaxation can be understood as requiring the resolution of a multi-marginal transportation   with an infinite number of marginals. Following~\cite{BrenierGeneralized} (equation (21) section VII), when discretizing this problem with $K$ steps in time, one thus faces the resolution of a $K$ marginals OT problem.

We consider a fixed uniform discretization of $[0,1]^d$ with points $(x_i)_{i=1}^N$. The marginals are the uniform measure on this set (as discretization of the Lebesgue measure), i.e. $\km{p}=\ones/N$ for all $k$. The prescribed volume preserving maps $\Xi : [0,1]^d \rightarrow [0,1]^d$ is discretized using a permutation of the grid points, i.e. a discrete bijection $\si : \{1,\ldots,N\} \rightarrow \{1,\ldots,N\}$.

The cost function is then
\eq{
	C_{j_1,\ldots,j_K}  = \sum_{k=1,\ldots,K-1} 
	\norm{ x_{j_{k+1}} - x_{j_{k}}  }^2   
	+  
	\norm{ x_{\si(j_1)} - x_{j_K} }^2 
}
and the optimal coupling $\pi$ solves~\eqref{eq-multimarg-entropic}.

For each $k \in \{1,\ldots,K\}$, the transition probability from "time"  $1$ to "time" $k$~: $T_{1,k} \in \RR^{N \times N}$ is defined as 
\begin{equation}
\label{2coupl}
	\foralls (s,w) \in \{1,\ldots,N\}^2, \quad
	(T_{1,k})_{s,w} = \sum_{j_i \ne {j_1,j_k}} \pi_{s,j_2,\ldots,j_{k-1},w,j_{k+1},\ldots,j_K}.
\end{equation}
It represents the evolution of a generalized flow of particles at  time $t=\frac{k-1}{K-1}$.  Note that in this setting, particles trajectories are 
non deterministic and their mass may split and spread across the domain. 


Brenier's numerical method~\cite{BrenierGeneralized} is based on an approximation of the measure preserving map by a one to one permutation of the domain and the representation of the diffuse coupling therefore needs a large number of particles. Our resolution method is different and computes a space discretization of the coupling matrix and naturally encodes non-diffeomorphic volume preserving maps. The coupling $\pi$ is an array of size $(N^d)^K$ where the $d$-dimensional physical domain is discretized on $N^d$ points and we have $K$ times steps. However, as explained in Remark~\ref{remjd} below, because of the structure of the cost we only need to store and multiply $(N^d)^2$ matrices. 
 
\begin{rem}\label{remjd}[Reduction to Transitions probabilities ]
As defined in~\eqref{eq-multimarginal-proj}, the resolution of the regularized $K$-marginal OT problem boils down to the computation of a KL projection of $\xi = e^{-\frac{C}{\ga}}$
We can rewrite the coupling using only 2 smaller matrices $\xi^0, \xi^1 \in \RR^{N \times N}$ since
\begin{equation}\label{gabarExp}
 	\xi_{j_{1},\ldots,j_{K}} = \pa{ \prod_{k=1}^{K-1} \xi^0_{j_{k},j_{k+1}} } \,  \xi^1_{j_{K}\si(j_{1})}
\end{equation}
\eq{
	\qwhereq
 	\xi^{0}_{\alpha,\beta} = 
	e^{-\frac{D_{\alpha \beta}}{\ga} },
	\quad\quad 
	\xi^{1}_{\beta,\alpha} = e^{-\frac{D_{\beta\si(\alpha)}}{\ga}}.
}
and $D_{\alpha \beta}=\norm{x_{\alpha}-x_{\beta} }^{2}$. 
We recall that all the marginals are equal to (a discretization of) the Lebesgue measure and $(x_{i})_i$ are discretized on the unit cube $[0,1]^{d}$.
As already noticed in Remark~\ref{rem-fast-sink},  the iterative Bregman projections~\eqref{eq-iter-bregmanproj} (always on the same Lebesgue marginal constraint)  can be simplified as an IPFP iterative procedure.
The optimal coupling $\pi$ that solves~\eqref{eq-multimarg-entropic} can be actually written as follows
\eq{
        \pi_{j_1,\ldots,j_K} = \xi_{j_1,\ldots,j_K} \prod_{k=1}^{K} u^k_{j_k} ,
}
where $u^k_{j_k}$ is the $(j_k)^{\text{th}}$ component of $u^k \in \RR^N$, for $k=1,\ldots,K$ (the $k^{\text{th}}$ vector being associated to the $k^{\text{th}}$ marginal).
Moreover the $u^k$ are uniquely determined by the constraints over the marginals of $\pi$
\eq{
   	u^k_{j_k} = \frac{1/N}{  
      	\sum_{m \neq k}  \sum_{j_m}  
		\{  \xi_{j_1,\ldots,j_K}   \prod_{\ell \neq k} u^\ell_{j_\ell}   \} 
	}
}
and the IPFP procedure for $K$ marginals is 
\eq{\label{IPFPmultimarg}
   u^{k,(n)}_{j_k} = \frac{1/N}{
   		\sum_{m \neq k} {\sum_{j_m}}  
		\{   \xi_{j_1,\ldots,j_K}  \prod_{\ell \neq k} {u^{\ell,(n-1)}_{j_\ell}}  \} 
	}
}
However, one can notice that the sum in~\eqref{IPFPmultimarg} could be computationally onerous, but thanks to~\eqref{gabarExp} we can rearrange it as   
\eq{
  	u^{k,(n)}_{j_k} =\frac{1/N}{ B_{j_k,j_k}}
}
where $u^{k,(n)}$ is the $k^{\text{th}}$ vector at step $(n)$ and $B$ is the the product of $K$ smaller $N \times N$ matrices 
\eq{
 	B=\xi^{\text{init}} \bigotimes_{\ell=1}^{K-1} \tilde{\xi}_{\ell}	
	\qwithq
	\xi^{\text{init}} = \choice{
		\xi^{0} \qforq k \neq K, \\
		\xi^{1} \quad \text{otherwise}
		}
}
\eq{
	\qandq
	\tilde{\xi}_{\ell} = \choice{
		\diag(u^{\sigma^k(\ell),(n)}) \otimes \xi^0 \qforq \sigma^k(\ell) \neq K \qandq \sigma^k(\ell) < k, \\
		\diag(u^{\sigma^k(\ell),(n-1)}) \otimes \xi^0 \qforq \sigma^k(\ell) \neq K \qandq \sigma^k(\ell) > k, \\
		\diag(u^{K,(n-1)})\otimes\xi^1 \quad \text{otherwise},
	}
}
where $\otimes$ is the standard matrix product and we use the Matlab convention that $\diag$ of vector is the diagonal matrix with vector values and $\diag$ of matrix is the vector of its diagonal.
We also highlight that, to use the simplification the  
 $u^l$ must be ordered in the correct way so that the computation of the sum for the  $k$ieth update starts at $k+1$ and finishes at $k-1$.
 We have introduced the circular  permutations $\sigma^{k}(\ell)=(\ell+k-1 \mod K)+1$ which 
returns the $(\sigma^k(\ell))^{\text{th}}$ term at the $\ell^{\text{th}}$ position of the product.

Each iteration of the IPFP procedure therefore only involves $(2K)$ 2-coupling matrices multiplications and only requires 
storing $K$ vectors and the 2-coupling cost matrices  $\xi^0$ and  $\xi^1$. 
The computation of the 2-coupling maps \eqref{2coupl} can be simplified with the same remark. 
\end{rem}

Figures~\ref{fig:mm1}, \ref{fig:mm2} and~\ref{fig:mm3} show $T_{1,k}$ for three test cases in dimension $d=1$ proposed in~\cite{BrenierGeneralized}. The computation is performed with a uniform discretization $(x_i)_i$ of $[0,1]$ with $N=200$ points, $\ga=10^{-3}$ and $K=16$. They agree with the solutions produced by Brenier and the mass spreading of the generalized flow is nicely captured by the 2 marginals couplings~\eqref{2coupl}.

\newcommand{\MyFigEuler}[2]{\imgbox{\includegraphics[width=.19\linewidth]{multimarginal/EulerT#1/EulerT#1timeStep#2}}}

\begin{figure}[h!]
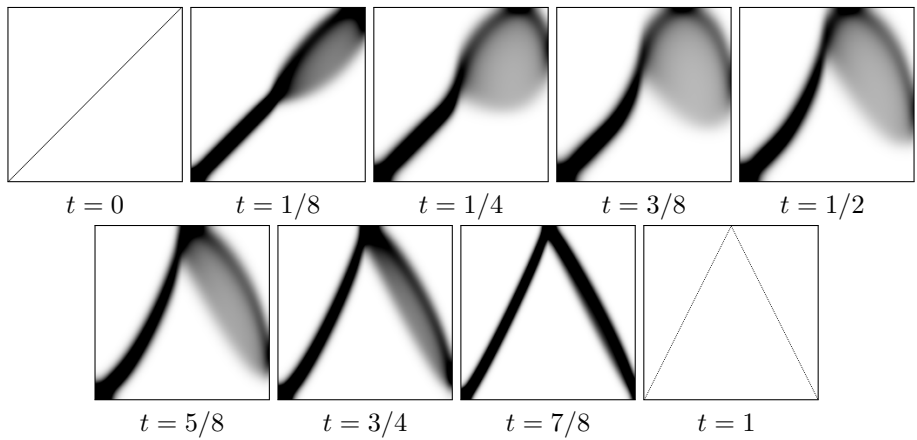

	\centering
	\TabFive{
		\MyFigEuler{1}{1} &
		\MyFigEuler{1}{2} &
		\MyFigEuler{1}{3} &
		\MyFigEuler{1}{4} &
		\MyFigEuler{1}{5} \\
		$t=0$ & $t=1/8$ & $t=1/4$ & $t=3/8$ & $t=1/2$
	}
	\TabFour{
		\MyFigEuler{1}{6} &
		\MyFigEuler{1}{7} &
		\MyFigEuler{1}{8} &
		\MyFigEuler{1}{9} \\
		$t=5/8$ & $t=3/4$ & $t=7/8$ & $t=1$
	}
	\caption{%
		Display of $T_{1,k}$ showing the evolution of the fluid particles from $x$ to $\Xi(x)=\min(2x,2-2x)$ for $x \in [0,1]$. The corresponding time is $t=\frac{k-1}{K-1} \in [0,1]$.
	}
   \label{fig:mm1}
\end{figure}

\begin{figure}[h!]
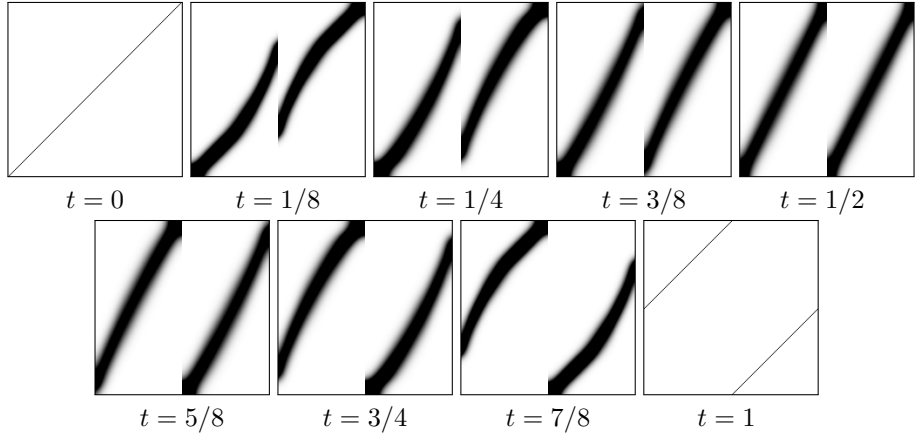

	\centering
		\TabFive{
		\MyFigEuler{2}{1} &
		\MyFigEuler{2}{2} &
		\MyFigEuler{2}{3} &
		\MyFigEuler{2}{4} &
		\MyFigEuler{2}{5} \\
		$t=0$ & $t=1/8$ & $t=1/4$ & $t=3/8$ & $t=1/2$
	}
	\TabFour{
		\MyFigEuler{2}{6} &
		\MyFigEuler{2}{7} &
		\MyFigEuler{2}{8} &
		\MyFigEuler{2}{9} \\
		$t=5/8$ & $t=3/4$ & $t=7/8$ & $t=1$
	}
	\caption{%
		Same as Figure~\ref{fig:mm1} for the map $\Xi(x)=(x+1/2) \text{ mod } 1$ for $x\in [0,1]$. 
	}
   \label{fig:mm2}
\end{figure}

\begin{figure}[h!]
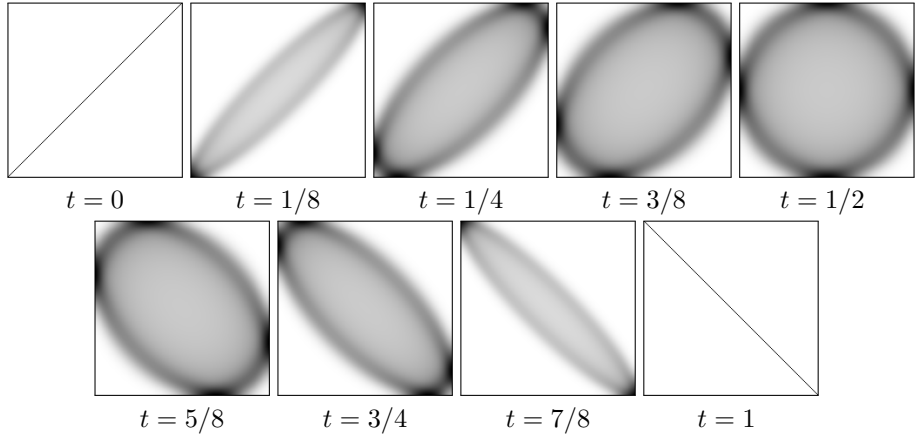

	\centering
		\TabFive{
		\MyFigEuler{3}{1} &
		\MyFigEuler{3}{2} &
		\MyFigEuler{3}{3} &
		\MyFigEuler{3}{4} &
		\MyFigEuler{3}{5} \\
		$t=0$ & $t=1/8$ & $t=1/4$ & $t=3/8$ & $t=1/2$
	}
	\TabFour{
		\MyFigEuler{3}{6} &
		\MyFigEuler{3}{7} &
		\MyFigEuler{3}{8} &
		\MyFigEuler{3}{9} \\
		$t=5/8$ & $t=3/4$ & $t=7/8$ & $t=1$
	}
	\caption{%
		Same as Figure~\ref{fig:mm1} for the map $\Xi(x)=1-x$ for $x\in [0,1]$. 
	}
   \label{fig:mm3}
\end{figure}


\section{Transport Problems with Ine\-quali\-ty Cons\-traints}

In this section, we consider transport problems with inequality constraints. Again we have to project for the KL divergence on the intersection of convex sets of nonnegative vectors.

\subsection{Partial Transport}
\label{sec-partial-ot}

In the partial transport problem, one is given two marginals $(p,q) \in (\RR_+^N)^2$, not necessarily with the same total mass. 
We wish  to transport only  a given fraction of mass 
\eq{
	m \in [0,\min( p^T\ones, q^T\ones )],  
}
minimizing the transportation cost $\dotp{C}{\pi}$ where $C \in (\RR_+)^{N \times N}$ is the ground cost. 

The corresponding regularized problem reads
\eql{\label{partialdiscrete}
	\min_{\pi\in \RR_+^{N\times N}}  
	\enscond{ 
			\dotp{C}{\pi} - \ga E(\pi)
		}{ 
			\pi \ones \leq p, \; 
			\transp{\pi} \ones \leq q, 
			\transp{\ones} \pi \ones = m 
		}
}
where the inequalities should be understood component-wise.

Similarly to~\eqref{eq-regul-ot-kl}, this is equivalent to computing the projection of $\xi=e^{-\frac{C}{\ga}}$ on the intersection $\Cc_1 \cap \Cc_2 \cap \Cc_3$ of $K=3$ convex sets where
\eql{\label{defdesCkpartial}
	\Cc_1 \eqdef \enscond{ \pi }{ \pi \ones \leq p }, \quad
	\Cc_2 \eqdef \enscond{ \pi }{ \transp{\pi} \ones \leq q }, \quad
	\Cc_3 \eqdef \enscond{  \pi }{ \transp{\ones} \pi \ones = m  }.
}
The following proposition shows that the KL projection onto those three sets can be obtained in closed form. 

\begin{prop}
Let   $\pi\in \RR_{+}^{N \times N}$. Denoting $\pi^k \eqdef \KLproj_{\Cc_k}( \pi)$ for $k \in \{1, 2, 3\}$ where $\Cc_k$ is defined by \eqref{defdesCkpartial}, one has
\begin{align*}
	\pi^1 &= \diag\pa{ \min\pa{ \frac{p}{\pi \ones}, \ones} } \pi, \\
	\pi^2 &= \pi \diag\pa{ \min\pa{ \frac{q}{\transp{\pi} \ones}, \ones} }, \\
	\pi^3 &= \pi \frac{m}{\transp{\ones} \pi \ones}, 
\end{align*}
where the minimum is component-wise.
\end{prop}

Since the considered sets $\Cc_1$ and $\Cc_2$ are convex but not affine, one thus needs to use Dykstra iterations~\eqref{eq-iter-dystra} which are ensured to converge to the solution of~\eqref{partialdiscrete}.

If $\pi^\star$ is the optimal solution of~\eqref{partialdiscrete} and 
\eq{
	p_m \eqdef \pi^\star \ones
	\qandq
	q_m \eqdef \pi^{\star,T} \ones
} 
are  its marginals, then we define the active source $\Ss_m$ and the active target $\Tt_m$ regions as follow
\begin{align*}
     \Ss_m &\eqdef \enscond{x_i }{  (p_m)_i/m \geq \eta }, \\
     \Tt_m &\eqdef \enscond{x_i }{  (q_m)_i/m \geq \eta },
\end{align*}
where $\eta>0$ is a threshold we use to detect the region, namely the active region, where the transported mass is concentrated.  

The continuous partial optimal transport problem has been studied in Caf\-far\-el\-li-McCann \cite{MR2630054} and 
Figalli \cite{MR2592287}. They show in particular that if there exists an hyperplane separating the support of the 
two marginals then the  ``active region" is separated from the ``inactive region" by a free boundary which can be parameterized as a semi concave graph over the
separating hyperplane. This can be observed on the test case presented in Figure~\ref{fig:pt1}.
The computation is performed on an uniform 2D-grid of $N=256\times256$ points in $[0,1]^{2}$, $\ga=10^{-3}$ and $m=0.7\min( \dotp{p}{\ones}, \dotp{q}{\ones} )$.
\begin{figure}[htdp]
	\centering
	\begin{tabular}{@{}c@{\hspace{1mm}}c@{\hspace{1mm}}c@{}}
		\imgbox{\includegraphics[width=.3\textwidth]{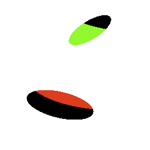}} & 
		\imgbox{\includegraphics[width=.3\textwidth]{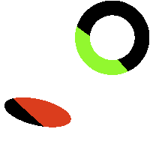}} & 
		\imgbox{\includegraphics[width=.3\textwidth]{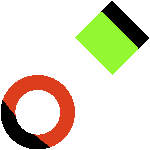}} 
	\end{tabular}
  	\caption{ 		 
		The red region is the active source $\Ss_m$, the green region is the active target $\Tt_m$ and the black ones are the inactive regions.
	} 
  	\label{fig:pt1}
\end{figure}



\subsection{Capacity Constrained Transport}
\label{sec-capacity-ot}

Korman and McCann proposed and studied in~\cite{km1,km2} a variant of the classical OT problem when there is an upper bound on the coupling weights so as to capture transport capacity constraints. 

The capacity is described by $\th \in (\RR_+)^{N \times N}$, where $\th_{i,j}$ is the maximum possible mass that can be transferred from $i$ to $j$. The corresponding regularized problem reads, for a ground cost $C \in (\RR_+)^{N \times N}$ and marginals $(p,q) \in (\RR_+^N)^2$, 
\eql{\label{constraineddiscrete}
	\min_{\pi\in \RR_+^{N\times N}}  
		\enscond{
			\dotp{C}{\pi} - \ga E(\pi)
		}{
			\pi\ones = p, \; 	
			\transp{\pi}\ones = q, \; 
			\pi \leq \th
		}
}
where the inequalities should be understood component-wise.

This problem is equivalent to a KL projection problem of type~\eqref{proj-inter} with $K=3$ convex sets and
\eql{\label{defdesCkconstrained}
	\Cc_1 \eqdef \enscond{ \pi }{ \pi\ones = p }, \quad
	\Cc_2 \eqdef \enscond{ \pi }{ \transp{\pi}\ones = q },   \quad	
	\Cc_3 \eqdef \enscond{ \pi }{ \pi \leq \th }.
}
The projection on $\Cc_1$ and $\Cc_2$ is given by Proposition~\ref{prop-projkl-row-cols}. The projection on $\Cc_3$ is simply 
\eq{
	\KLproj_{\Cc_3}(\pi) = \min(\pi, \theta)
}
where the minimum is component-wise.


Korman and McCann~\cite{km2} established several interesting properties of minimizers in the  continuous setting. In particular, they proved  (theorem 3.3) that optimal plans must saturate the constraint $\Cc_3$: 
the optimal $\pi$  is of the form $\th\, 1_W$, where $1_W$ is the characteristic function of a subset $W$ of $\RR^d \times \RR^d$.
They also prove in lemma 4.1 \cite{km2} symmetries properties between minimizers  $\pi^\star$, $\tilde\pi^{\star}$ with the same marginals but  different capacity constraints  $\th$ and $\tilde\th$  which are H\"older conjugate, i.e. $\frac{1}{\th} + \frac{1}{\tilde\th} = 1$. More precisely, assuming for simplicity that the marginals are symmetric with respect to 0, they show that 
\begin{equation} 
	\label{sym}
	\pi^\star  = \th\, 1_W   \iff
	\tilde\pi^\star  = \tilde\th\, 1_{R( \Omega \setminus W)}
\end{equation} 
 where $R(x,y) = (x,-y)$ is the symmetry with respect to the second marginal axis and $W$ the optimal support of the saturated constraint 
 $\Cc_3$.

\newcommand{\figCapacity}[1]{\imgbox{\includegraphics[width=.28\textwidth]{capacity/#1}}}

\begin{figure}[htdp]
	\centering
	\TabThree{
		\figCapacity{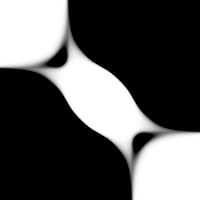} & 
		\figCapacity{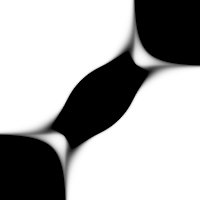} & 
		\figCapacity{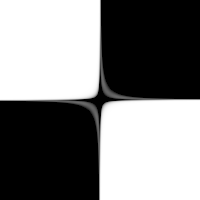} \\
		$\theta=3/2$ & $\theta=3$ & $\th=2$
	}
  	\caption{ 		 
		Comparison of optimal coupling $\pi^\star$ for different values of $\th$. 		
	} 
  	\label{fig:cc1}
\end{figure}


\begin{figure}[htdp]
	\centering
	\TabThree{	
		\figCapacity{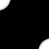} &
		\figCapacity{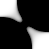} &
		\figCapacity{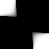} \\		
		\figCapacity{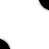} &
		\figCapacity{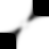} &
		\figCapacity{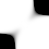}  \\
		$(i_2,j_2)=(\frac{\sqrt{N}}{4},\frac{\sqrt{N}}{4})$ &
		$(i_2,j_2)=(\frac{\sqrt{N}}{2},\frac{\sqrt{N}}{2})$ &
		$(i_2,j_2)=(\frac{3\sqrt{N}}{4},\frac{3\sqrt{N}}{4})$ \\
	}
  	\caption{ 
		2-D slices of the optimal coupling $\pi^\star$ 
		of the form $(\pi^\star_{(i_1,i_2)(j_1,j_2)})_{i_1,j_1}$, each time for 
		some fixed value of $(i_2,j_2) \in \{1,\ldots,\sqrt{N}\}^2$, 
		for $\th=3/2$ (top row) and $\th=3$ (bottom row).
	} 
  	\label{fig:cc3}
\end{figure}

Korman and McCann~\cite{km2} illustrated their theory with two 1-D numerical test cases (Figures 1 and 2 of~\cite{km2}) computed by linear programming and a discretization of the problem on a cartesian grid. We tested our method on the same tests. 
The ground cost is the standard quadratic distance $C_{i,j} = \norm{x_i-x_j}^2$, the marginals are discretizations of the uniform distribution on $[-1/2,1/2]$ using $N=100$ points $(x_i)_{i}$.
The simulation uses $\ga=10^{-3}$.
We reproduce the expected symmetries~\eqref{sym} in Figure~\ref{fig:cc1} in the 1-D test case for $\th=\frac{3}{2}$  and $\tilde\th=3$ and also for the self dual H\"older conjugate $\th=\tilde\th=2$.   


We also computed the solutions of similar test cases but this time in 2-D, which would be computationally too expensive to solve with linear programming methods.
The marginals $(p,q)$ are discretization of the uniform distribution on the square $[-1/2,1/2]^2$, discretized on a grid of $N=50 \times 50$ points $(x_i)_i$.
The simulation uses $\ga=10^{-3}$.
Figure~\ref{fig:cc3} shows some slices of the 4-D array representing the optimal transport plan $\pi^\star$, $\tilde\pi^\star$, illustrating the symmetries~\eqref{sym} in this setting.


\subsection{Multi-Marginal Partial Transport}

In~\cite{passkitagawa} Pass and Kitagawa studied the multi-marginal partial transport problem and, as a natural extension, the partial barycenter problem. Let us consider $K$ marginals $(p_k)_{k=1}^{K}$ and a transport plan $\pi \in \RR_{+}^{N^K}$.
We now combine the (regularized) partial optimal transport and the ``standard" multi-marginal problem, as described
in Sections~\ref{sec-partial-ot} and~\ref{subsec-multi-marginal-regularization} respectively. We obtain the following problem
\eql{\label{eq-partial-multimarginal-proj}
	\min_\pi \enscond{
		\KLdiv{\pi}{\xi} 
	}{
		\pi \in \Cc_1 \cap \Cc_2 \cap \ldots \cap \Cc_{K+1}
	}
	\qwhereq
	\xi \eqdef e^{-\frac{C}{\ga}}
}
where

\begin{align*}
	\Cc_k &\eqdef \enscond{ \pi \in \RR_+^{N^K} }{ S_k(\pi) \leq \km{p} }\quad k=1,\ldots,K, \\
	\Cc_{K+1} &\eqdef \enscond{ \pi \in \RR_+^{N^K} }{\textstyle \sum_{j}\pi_{j}=m},
\end{align*}
with $m \in [0,\min_k(\dotp{p_k}{\ones})]$

The KL projections on these convex sets are detailed in the following proposition.

\begin{prop}
	For any $k = 1,\ldots,K+1$, denoting $\pi^k = \KLproj_{\Cc_k}(\bar\pi)$, one has 
	\eq{
		\foralls k=1,\ldots,K, \; 
		\foralls j=(j_1,\ldots,j_K), \quad		
		\pi_{j}^k = \min\pa{ \frac{(\km{p})_{j_k}}{ S_k(\bar\pi)_{j_k} },1 } \bar\pi_{j} 
	}
	\eq{		
		\pi^{K+1} =\frac{m}{ \sum_{j}\bar\pi_{j} } \bar\pi. 
	}
\end{prop}

Once again the sets $\Cc_k$ are not affine, so  one  needs to use Dykstra iterations~\eqref{eq-iter-dystra}.

Figure~\ref{fig:pmm1} shows the results obtained when solving~\eqref{eq-partial-multimarginal-proj} with the same three marginals $(p_1,p_2,p_3)$ used in Figure~\ref{fig-barycenter-entropic}, using the cost 
\eql{\label{partial-quadratic-standard}
	C_{j_1,j_2,\ldots,j_K}  = \sum_{1 \leq s,t \leq K}\dfrac{1}{2} \norm{ x_{j_s}-x_{j_t} }^2, 
}
The computation is performed on an uniform 2D-grid of $N=60\times60$ points in $[0,1]^2$, $\ga=0.005$ and $m=0.7\min_k(\dotp{p_k}{\ones})$.

%

\newcommand{\figMultiPartial}[1]{\imgbox{\includegraphics[width=.4\textwidth]{multimarginal/PartialMultimarginal/#1}}}

\begin{figure}[htdp]
	\centering
		\figMultiPartial{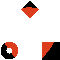} 
%
  	\caption{
		Multi-marginal partial transport.
		The active regions are displayed in red. 
	} 
  	\label{fig:pmm1}
\end{figure}


\section*{Conclusion}

In this paper, we have presented a unifying framework to approximate solutions of various OT-related linear programs through entropic regularization. This regularization enables the use of simple, yet powerful, iterative KL projection methods. While the entropy penalization is not a competitor with interior point methods when it comes to accurately solve the initial linear program, it produces fast approximations at the expense of an extra smoothing. It is thus a method of choice for many applications such as machine learning, image processing or economics.

\section*{Aknowledgements}

We would like to thank Yann Brenier and Brendan Pass for stimulating discussions.
The work of G. Peyr\'e has been supported by the European Research Council (ERC project SIGMA-Vision).
JD. Benamou, G. Carlier and L. Nenna gratefully acknowledge the support of the ANR, through the project ISOTACE (ANR-12-MONU-0013) and INRIA through the ``action exploratoire" MOKAPLAN.
M. Cuturi gratefully acknowledges the support of JSPS young researcher A grant 26700002 and the gift of a K40 card from the NVIDIA corporation.

\bibliographystyle{plain}  
\bibliography{refs}

\end{document}